\documentclass{amsart}
\title
[On the Cyclicity of the Unramified Iwasawa modules]
{
On the Cyclicity of the Unramified Iwasawa Modules of the Maximal
Multiple $\mathbb{Z}_p$-Extensions Over Imaginary Quadratic Fields
}
\author{Takashi MIURA, Kazuaki MURAKAMI, Keiji OKANO and Rei OTSUKI}
\date{}
\usepackage{amsmath,amscd,amssymb,amsthm,euscript,ascmac}
\usepackage{amsfonts}
\usepackage{latexsym}
\usepackage{amscd}
\usepackage{amsthm}
\usepackage[dvips]{graphicx}
\usepackage{layout}
\usepackage{array}
\usepackage{tabularx}
\usepackage{multicol}
\usepackage{bm}
\usepackage{url}
	\usepackage[dvipdfmx]{color}
\input cyracc.def

\newtheorem{Example}{Example}

\newtheorem{thm}{Theorem}[section]
\newtheorem{prop}[thm]{Proposition}
\newtheorem{cor}[thm]{Corollary}

\newtheorem{lem}[thm]{Lemma}
\newtheorem{Rem}[thm]{Remark}

\begin{document}
\begin{abstract}
For an odd prime number $p$, we study the number of generators of the unramified Iwasawa modules of the maximal multiple $\mathbb{Z}_p$-extensions over Iwasawa algebra.
In a previous paper of the authors, under several assumptions for an imaginary quadratic field, we obtain a necessary and sufficient condition for the Iwasawa module to be cyclic as a module over the Iwasawa algebla.
Our main result is to give methods for computation and numerical examples about the results.
We remark that our results do not need the assumption that Greenberg's generalized conjecture holds.
\end{abstract}
\maketitle


\section{Introduction}\label{Introduction}
Let $p$ be a prime number, $\mathbb{Z}_p$ the ring of $p$-adic integers, 
$K$ an algebraic number field of finite degree,
and $K_\infty^{\rm c}$ the cyclotomic $\mathbb{Z}_p$-extension of $K$.
One of the most important objects in classical Iwasawa theory is the Galois group $X_{K_\infty^{\rm c}}$ of the maximal unramified abelian pro-$p$ extension of $K_\infty^{\rm c}$.
The Galois group ${\rm Gal}(K_\infty^{\rm c}/K)$ acts on $X_{K_\infty^{\rm c}}$ by the inner product, 
and it is well known that $X_{K_\infty^{\rm c}}$ is a finitely generated torsion $\mathbb{Z}_p[[{\rm Gal}(K_\infty^{\rm c}/K)]]$-module.
We introduce here a well-known case below where we can reduce computing the number of generators $X_{K_\infty}^{\rm c}$
to computing that of the $p$-Sylow subgroup $A_K$ of the ideal class group of $K$.
If $p$ does not split in $K$ and is totally ramified in $K_\infty^{\rm c}/K$, 
then the ${\rm Gal}(K_\infty^{\rm c}/K)$-coinvariant of $X_{K_\infty^{\rm c}}$ is isomorphic to $A_K$, 
and hence Nakayama's lemma tells us that
the number of generators of $X_{K_\infty^{\rm c}}$ as a $\mathbb{Z}_p[[{\rm Gal}(K_\infty^{\rm c}/K)]]$-module coincides with $\dim_{\mathbb{F}_p}(A_K/pA_K)$ (see \cite[Proposition 13.22]{Wa}).
In particular,
$X_{K_\infty^{\rm c}}$ is cyclic as a $\mathbb{Z}_p[[{\rm Gal}(K_\infty^{\rm c}/K)]]$-module 
if and only if $A_K$ is cyclic as an abelian group.

The objective of our study is to generalize these basic facts to the case of multiple $\mathbb{Z}_p$-extensions.
In other words, for the Galois group $X_{\widetilde K}$ of 
the maximal unramified abelian pro-$p$ extension of the maximal multiple $\mathbb{Z}_p$-extension $\widetilde{K}$ of $K$,
we aim for describing the number of generators of $X_{\widetilde{K}}$ 
as a $\mathbb{Z}_p[[{\rm Gal}(\widetilde{K}/K)]]$-module, 
and also for giving conditions 
that $X_{\widetilde{K}}$ is to be $\mathbb{Z}_p[[{\rm Gal}(\widetilde{K}/K)]]$-cyclic.
There is an important conjecture called Greenberg's generalized conjecture,
which states that $X_{\widetilde{K}}$ would be pseudo-null as a $\mathbb{Z}_p[[{\rm Gal}(\widetilde{K}/K)]]$-module.
A lot of evidences supporting the validity of the conjecture have been found.
However,
this conjecture does not imply the number of generators of $X_{\widetilde K}$.
Therefore, it is worthwhile 
to describe the number of generators of $X_{\widetilde K}$ as a $\mathbb{Z}_p[[{\rm Gal}(\widetilde K/K)]]$-module, 
to give the necessary and sufficient condition for $X_{\widetilde K}$ to be $\mathbb{Z}_p[[{\rm Gal}(\widetilde K/K)]]$-cyclic, 
to provide these numerical examples and so on. 
We expect that these studies will help us to gain a deeper understanding of other various properties of $X_{\widetilde K}$.

In \cite{MOMO}, the authors give some conditions that $X_{\widetilde{K}}$ is to be $\mathbb{Z}_p[[{\rm Gal}(\widetilde{K}/K)]]$-cyclic for imaginary quadratic fields $K$.
In this paper, we give methods for computation and examples about the results.
In the rest of this section, we prepare notation
and introduce the theorems in \cite{MOMO} (Theorems \ref{thm 3}, \ref{main thm of classification lambda=2}).
In \S 2, we give a method of computation and examples about Theorem \ref{thm 3}.
In \S 3, we introduce Sumida's and Koike's results which is a classification of Iwasawa modules of $\mathbb{Z}_p$-rank $2$.
In \S 4, we give a method of computation and examples about Theorem \ref{main thm of classification lambda=2}.

\subsection{Conditions for $X_{\widetilde{K}}$ to be $\mathbb{Z}_p[[{\rm Gal}(\widetilde{K}/K)]]$-cyclic}\label{notation and previous theorems}
Let $p$ be an odd prime number, $K$ an imaginary quadratic field in which $p$ does not split.
Denote by $K_\infty^{{\rm c}}$ and $K_\infty^{{\rm an}}$ the cyclotomic $\mathbb{Z}_p$-extension and the anti-cyclotomic $\mathbb{Z}_p$-extension of $K$, respectively.
Put $\widetilde{K}=K_\infty^{{\rm c}}K_\infty^{{\rm an}}$.
Then $\widetilde{K}$ is the maximal multiple $\mathbb{Z}_p$-extension over $K$ and ${\rm Gal}(\widetilde{K}/K) \cong \mathbb{Z}_p^2$.
Fix a topological generator $\widetilde{\sigma}$ (resp. $\widetilde{\tau}$) of ${\rm Gal}(\widetilde{K}/K_\infty^{{\rm an}})$ (resp. ${\rm Gal}(\widetilde{K}/K_\infty^{{\rm c}})$).
Then there is a ring isomorphism between the complete group ring $\mathbb{Z}_p[[{\rm Gal}(\widetilde{K}/K)]]$ and the formal power series ring $\mathbb{Z}_p[[S,T]]$ by sending 
$\widetilde{\sigma}$ and $\widetilde{\tau}$ to $1+S$ and $1+T$, respectively.
Note that it depends on the choice of topological generators $\widetilde{\sigma}$ and $\widetilde{\tau}$.
Also, we have a commutative diagram
$$
\begin{CD}
\mathbb{Z}_p[[{\rm Gal}(\widetilde{K}/K)]] @> \sim >> \mathbb{Z}_p[[S,T]]
\\
@VVV  @VVV
\\
\mathbb{Z}_p[[{\rm Gal}(K_\infty^{{\rm c}}/K)]] @> \sim >> \mathbb{Z}_p[[S]]
\end{CD},
$$
where the left vertical arrow is induced by the projection
${\rm Gal}(\widetilde{K}/K) \to {\rm Gal}(K_\infty^{{\rm c}}/K)$
and the right vertical arrow is defined by substituting $T=0$.
We identify $\mathbb{Z}_p[[{\rm Gal}(\widetilde{K}/K)]]$ (resp. $\mathbb{Z}_p[[{\rm Gal}(K_\infty^{{\rm c}}/K)]]$) with $\mathbb{Z}_p[[S,T]]$ (resp. $\mathbb{Z}_p[[S]]$) via the isomorphism above.
For any algebraic field $F$, denote by $X_F$ the Galois group of the maximal unramified abelian pro-$p$ extension $L_F$ of $F$.
If $F$ is a finite extension of the rational number field $\mathbb{Q}$, denote by $A_F$ the $p$-Sylow subgroup of the ideal class group of $F$.

It is known that, for any $\mathbb{Z}_p$-extension $K_\infty$ of $K$, $X_{K_\infty}$ is a finitely generated torsion $\mathbb{Z}_p[[{\rm Gal}(K_\infty/K)]]$-module.
Similarly, $X_{\widetilde{K}}$ is a finitely generated torsion $\mathbb{Z}_p[[S,T]]$-module by Greenberg \cite{Greenberg73}.
Moreover, for the cyclotomic $\mathbb{Z}_p$-extension $K_\infty^{{\rm c}}$, $X_{K_\infty^{{\rm c}}}$ is a finitely generated free $\mathbb{Z}_p$-module by Ferrero and Washington \cite{F-W} and \cite[Proposition 13.28]{Wa}.
By Nakayama's lemma, the number of generators of 
$X_{K_{\infty}^{\rm c}}$ 
(resp. $X_{\widetilde{K}}$) 
as $\mathbb{Z}_p[[S]]$-module coincides with 
$\dim_{\mathbb{F}_p} X_{K_\infty^{{\rm c}}}/(p,S)X_{K_\infty^{{\rm c}}}$
(resp. $\dim_{\mathbb{F}_p} X_{\widetilde{K}}/(p,S,T)X_{\widetilde{K}}$).
Furthermore, since $p$ does not split in $K$, we have 
$$
\dim_{\mathbb{F}_p} X_{K_\infty^{{\rm c}}}/(p,S)X_{K_\infty^{{\rm c}}}
=
\dim_{\mathbb{F}_p}  A_K/pA_K.
$$

We introduce the Iwasawa invariants and the characteristic ideals.
Let $\mathcal{O}$ be the ring of integers of a finite extension over the field $\mathbb{Q}_p$ of $p$-adic numbers,
and $M$ a finitely generated torsion $\mathcal{O}[[S]]$-module.
By the structure theorem of $\mathcal{O}[[S]]$-modules, there is an $\mathcal{O}[[S]]$-homomorphism
$$
\varphi \colon
M \to
\left(\bigoplus_i \mathcal{O}[[S]]/(\pi^{m_i})\right) 
\oplus 
\left(\bigoplus_j \mathcal{O}[[S]]/(f_j(S)^{n_j})\right) 
$$
with finite kernel and finite cokernel, where $m_i$, $n_j$ are non-negative integers, $\pi$ is a prime element in $\mathcal{O}$, and
$f_j(S) \in \mathcal{O}[S]$ are distinguished irreducible polynomials. 
We put
$$
{\rm char}(M)=\left(\prod_i \pi^{m_i}\prod_j f_j(S)^{n_j} \right),
$$
which is an ideal in $\mathcal{O}[[S]]$ and called the characteristic ideal of $M$.
We define the Iwasawa $\mu$-invariant $\mu(K_\infty/K)$ and the Iwasawa $\lambda$-invariant $\lambda(K_\infty/K)$ of a $\mathbb{Z}_p$-extension $K_\infty$ by $\sum_i m_i$ and $\sum_j n_j \deg f_j$ for $M=X_{K_\infty}$, respectively.

Now we introduce the theorems in \cite{MOMO} which give conditions for $X_{\widetilde{K}}$ to be $\mathbb{Z}_p[[{\rm Gal}(\widetilde{K}/K)]]$-cyclic.

\begin{thm}{\rm (\cite[ Thoerem 1.1]{MOMO})}\label{thm 3}
\makeatletter
  \parsep   = 0pt
  \labelsep = .5pt
  \def\@listi{%
     \leftmargin = 20pt \rightmargin = 0pt
     \labelwidth\leftmargin \advance\labelwidth-\labelsep
     \topsep     = 0\baselineskip
     \partopsep  = 0pt \itemsep       = 0pt
     \itemindent = 0pt \listparindent = 10pt}
  \let\@listI\@listi
  \@listi
  \def\@listii{%
     \leftmargin = 20pt \rightmargin = 0pt
     \labelwidth\leftmargin \advance\labelwidth-\labelsep
     \topsep     = 0pt \partopsep     = 0pt \itemsep   = 0pt
     \itemindent = 0pt \listparindent = 10pt}
  \let\@listiii\@listii
  \let\@listiv\@listii
  \let\@listv\@listii
  \let\@listvi\@listii
  \makeatother
Let $p$ be an odd prime number and $K$ an imaginary quadratic field such that $p$ does not split.
\begin{itemize}
\item[{\rm (i)}]
{\rm (trivial case)}
Assume that $L_K \cap \widetilde{K}=K$, then 
$$
\dim_{\mathbb{F}_p} (X_{\widetilde{K}}/(p,S,T)X_{\widetilde{K}})=\dim_{\mathbb{F}_p} (A_K/pA_K).
$$
\item[{\rm (ii)}]
Suppose that $L_K \cap \widetilde{K} \neq K$, and that $\dim_{\mathbb{F}_p} (A_K/pA_K)=1$.
\begin{itemize}
\item[{\rm (ii-a)}]
If $\lambda(K_\infty^{{\rm c}}/K)=1$, then 
$
\dim_{\mathbb{F}_p} (X_{\widetilde{K}}/(p,S,T)X_{\widetilde{K}})=1.
$
\item[{\rm (ii-b)}]
If $\lambda(K_\infty^{{\rm c}}/K)\ge 2$, then
$$
\dim_{\mathbb{F}_p} (X_{\widetilde{K}}/(p,S,T)X_{\widetilde{K}})=
\begin{cases}
1\ \ \text{if $L_K \subset \widetilde{K}$},
\\
2\ \ \text{otherwise}.
\end{cases}
$$
\end{itemize}
\end{itemize}
\end{thm}

\begin{thm}{\rm (\cite[Thoerem 5.12]{MOMO})}\label{main thm of classification lambda=2}
Let $p$ be an odd prime number, $K$ an imaginary quadratic field such that $p$ does not split.
Suppose the following conditions:
\makeatletter
  \parsep   = 0pt
  \labelsep = .5pt
  \def\@listi{%
     \leftmargin = 20pt \rightmargin = 0pt
     \labelwidth\leftmargin \advance\labelwidth-\labelsep
     \topsep     = 0\baselineskip
     \partopsep  = 0pt \itemsep       = 0pt
     \itemindent = 0pt \listparindent = 10pt}
  \let\@listI\@listi
  \@listi
  \def\@listii{%
     \leftmargin = 20pt \rightmargin = 0pt
     \labelwidth\leftmargin \advance\labelwidth-\labelsep
     \topsep     = 0pt \partopsep     = 0pt \itemsep   = 0pt
     \itemindent = 0pt \listparindent = 10pt}
  \let\@listiii\@listii
  \let\@listiv\@listii
  \let\@listv\@listii
  \let\@listvi\@listii
  \makeatother
\ 
\begin{itemize}
\item
$\dim_{\mathbb{F}_p}(A_K/pA_K)=2$ and ${\rm Gal}(L_K \cap \widetilde{K}/K)$ is a direct summand of ${\rm Gal}(L_K/K)$.
\item
$\lambda(K_\infty^{{\rm c}}/K) =2$.
\item
Let $\alpha, \beta \in \overline{\mathbb{Q}_p}$ be the roots of the distinguished polynomial generating ${\rm char}(X_{K_\infty^{{\rm c}}})$.
Then $\alpha \neq \beta$.
\end{itemize}
We denote by ${\rm ord}$ the normalized additive valuation on $\mathcal{O}:=\mathbb{Z}_p[\alpha, \beta]$.
Put $m:=\min\{ {\rm ord}(\alpha),{\rm ord}(\beta) \}$.
Let $x_2 \in X_{K_\infty^{{\rm c}}}$ be a preimage of a generator of ${\rm Gal}(L_K/L_K \cap \widetilde{K})$.
Also, we denote 
by the vector
$
\begin{bmatrix}
\mu_{21}
\\
\mu_{22}
\end{bmatrix}
$
the image of $x_2 \otimes 1$ under the injective map
$$
X_{K_\infty^{{\rm c}}}\otimes_{\mathbb{Z}_p} \mathcal{O}
\to 
\mathcal{O}[[S]]/(S-\alpha) \oplus \mathcal{O}[[S]]/(S-\beta)
$$
defined in {\rm \S \ref{Sumida's and Koike's results}}.
Then, $X_{\widetilde{K}}$ is $\mathbb{Z}_p[[S,T]]$-cyclic if and only if one of the following holds:
$$
\hspace*{-4pt}
\begin{array}{llllll}
{\rm (i)} & 
k>0, & {\rm ord}(\beta-\alpha)-k< m, & 
\\
{\rm (ii)} & 
k>0, & {\rm ord}(\beta-\alpha)-k= m,&   {\rm ord}(\mu_{21})=0, 
\\ 
{\rm (iii)} & 
k=0, & {\rm ord}(\beta-\alpha)= m,\ n_1< n_2, & 
{\rm ord}(\mu_{21})=0,
\\
{\rm (iv)} & 
k=0, & {\rm ord}(\beta-\alpha)= m,\ n_1 \ge n_2, & 
\begin{cases}
{\rm ord}(\mu_{21})=0,
\\
{\rm ord}(\mu_{22})={\rm ord}(\beta)-{\rm ord}(\alpha),
\end{cases}
\end{array}
$$
where each $n_1$ and $n_2$ is defined by 
$p^{n_1}=\#{\rm Gal}(L_K \cap \widetilde{K}/K)$ and $p^{n_2}=\# {\rm Gal}(L_K/L_K \cap \widetilde{K})$, respectively.
\end{thm}
\section{examples of Theorem \ref{thm 3}}
In this section, we will give examples of Theorem \ref{thm 3}.
As in the previous section, let $p$ be an odd prime number, $K$ an imaginary quadratic field
in which $p$
does not split.
We denote by $ \frak{X}_{K}$ the Galois group of the maximal abelian pro-$p$ extension $M_{K}/K$
unramified outside
the prime lying above $p$.
Let $E_{K}$ be the unit group of $K$.
Put a prime  $\mathfrak{p}$ of $K$ lying above $p$.
$E_{K}^{(1)}=\{ u \in E_{K} ~|~u \equiv 1 \mathrm{~mod~}\frak{p} \}$.
By class field theory, we have the  following exact sequence
\[
0 \rightarrow \mathrm{Tor}_{\mathbb{Z}_p}\left( U_{\frak{p}}^{(1)} / \overline{\varphi(E_{K}^{(1)})}\right)
  \rightarrow  \mathrm{Tor}_{\mathbb{Z}_p} \frak{X}_{K}
  \rightarrow \mathrm{Gal}(L_{K}/L_{K} \cap \widetilde{K})
  \rightarrow  0,
\]
where $U_{\frak{p}}^{(1)}$ is the group of the principal units in the completion of $K$ with respect to $\mathfrak{p}$,
$\displaystyle{\varphi: E_{K}^{(1)} \rightarrow U_{\frak{p}}^{(1)}}$ is the natural homomorphism, and
$\overline{\varphi(E_{K}^{(1)})}$ is the closure of $\varphi(E_{K}^{(1)})$ in $U_{\frak{p}}^{(1)}$.
We know $L_{K} \cap \widetilde{K} \subset K_\infty^{{\rm an}} $.
Combining the exact sequence above with the following lemma, 
we can determine the integer $n$ such that $L_{K} \cap \widetilde{K}  = K_n^{{\rm an}}$.
\begin{lem}{\rm (Fujii {\cite[Lemma $4.3$]{Fu1}})}\label{Fujii}
Let $I_{K}(p)$ be the subgroup of the group of fractional ideals of $K$ prime to $p$ and
$S_{K}(p^{n})$ the Strahl group of $K$ modulo $p^n$, which consists of all fractional principal ideals $(\alpha)$
of $K$ satisfying $\alpha \equiv 1 ~~\mathrm{~mod~}p^n.$
Let $p^N = p\mathrm{~exp}(A_{K})$, where $\mathrm{exp}(A_{K})$ is the exponent of $A_{K}$.
If 
$$
(I_K(p)/S_K(p^n)) \otimes \mathbb{Z}_p \cong A \oplus \mathbb{Z}/p^{N_1} \mathbb{Z} \oplus \mathbb{Z}/p^{N_2} \mathbb{Z}
$$ 
for some 
integers $N_1,N_2$ satisfying $N+2\leq n,N<N_{i} ~(i=1,2)$, then we have $\mathrm{Tor}_{\mathbb{Z}_p} \frak{X}_{K} \cong A$, non-canonically.
\end{lem}
We also use the following criterion of whether $L_{K} \subset \widetilde{K}$.
\begin{lem}{\rm (Minardi \cite[Corollary of Proposition $6.\mathrm{B}$]{Min})}\label{brink}
Let $K=\mathbb{Q}(\sqrt{-d})$ with a square-free positive integer $d$.
If $p=3$ and $d \not\equiv 3 \mathrm{~mod~}9$, then
$L_K \subset \widetilde{K}$
if and only if 
the class number of $\mathbb{Q}(\sqrt{3d})$ is not divisible by $3$.
\end{lem}
Using Lemmas \ref{Fujii}, \ref{brink} above and referring Fukuda's table for the
$\lambda$-invariants of imaginary quadratic fields (\cite{Fukuda}),
we get the following examples.
\begin{Example}
\begin{rm}
Let $p=7$ and $K=\mathbb{Q}(\sqrt{-71})$.
Then the prime $7$ is inert in $K$.
In this case we have $\lambda(K_\infty^{{\rm c}}/K)=1$.
We can check that $A_K \cong \mathbb{Z}/7\mathbb{Z}$ and that $L_K \cap \widetilde{K}=K$ by Lemma \ref{Fujii}.
Hence $X_{\widetilde{K}}$ is cyclic as a $\mathbb{Z}_p[[\mathrm{Gal}(\widetilde{K}/K)]]$-module
by Theorem \ref{thm 3}(i).
\end{rm}
\end{Example}
\begin{Example}
\begin{rm}
Let $p=3$ and $K=\mathbb{Q}(\sqrt{-61})$.
Then the prime $3$ is inert in $K$.
In this case we have $\lambda(K_\infty^{{\rm c}}/K)=1$.
We can check that $A_K \cong \mathbb{Z}/3\mathbb{Z}$
and $L_K \subset \widetilde{K}$ by Lemma \ref{brink}.
Hence $X_{\widetilde{K}}$ is cyclic as a $\mathbb{Z}_p[[\mathrm{Gal}(\widetilde{K}/K)]]$-module
by Theorem \ref{thm 3}(ii-a).
\end{rm}
\end{Example}
\begin{Example}
\begin{rm}
Let $p=3$ and $K=\mathbb{Q}(\sqrt{-1207})$.
Then the prime $3$ is inert in $K$.
In this case we have $\lambda(K_\infty^{{\rm c}}/K)=2$.
We can check that $A_K \cong \mathbb{Z}/3^2\mathbb{Z}$ and that
$ L_K \subset \widetilde{K}$ by Lemma \ref{brink}.
Hence $X_{\widetilde{K}}$ is cyclic as a $\mathbb{Z}_p[[\mathrm{Gal}(\widetilde{K}/K)]]$-module
by Theorem \ref{thm 3}(ii-b).
\end{rm}
\end{Example}
\begin{Example}
\begin{rm}
Let $p=3$ and $K=\mathbb{Q}(\sqrt{-186})$.
Then the prime $3$ is ramified in $K$.
In this case we have $\lambda(K_\infty^{{\rm c}}/K)=2$.
We can check that $A_K \cong \mathbb{Z}/3\mathbb{Z}$ and that $L_K \subset \widetilde{K}$
by Lemma \ref{brink}.
Hence $X_{\widetilde{K}}$ is cyclic as a $\mathbb{Z}_p[[\mathrm{Gal}(\widetilde{K}/K)]]$-module
by Theorem \ref{thm 3}(ii-b).
\end{rm}
\end{Example}
\begin{Example}
\begin{rm}
Let $p=3$ and $K=\mathbb{Q}(\sqrt{-6382})$.
Then the prime $3$ is inert in $K$.
In this case we have $\lambda(K_\infty^{{\rm c}}/K)=2$.
We can check that $A_K \cong \mathbb{Z}/3^2\mathbb{Z}$ and that $K \neq L_K \cap \widetilde{K}$ and $ L_K \not\subset \widetilde{K}$.
Hence $X_{\widetilde{K}}$ is not cyclic as a $\mathbb{Z}_p[[\mathrm{Gal}(\widetilde{K}/K)]]$-module
by Theorem \ref{thm 3}(ii-b).
\end{rm}
\end{Example}
\section{Sumida's and Koike's results}\label{Sumida's and Koike's results}
In this section, we prepare notation for giving examples of Theorem \ref{main thm of classification lambda=2}.
Let $E$ be a finite extension over $\mathbb{Q}_p$.
Let $\mathcal{O}_{E},~\pi_E,$ and $\mathrm{ord}_{E}$
be the ring of integers in $E$, a prime element of $E$, and the normalized additive valuation 
on $E$ such that $\mathrm{ord}_E(\pi_E)=1$, respectively.
We put $\Lambda_{E}:=\mathcal{O}_{E}[[S]] $, the ring of formal power series over $\mathcal{O}_{E}$.
For a finitely generated torsion $\Lambda_{E}$-module $M$,
we denote the $\Lambda_E$-isomorphism class of $M$ by $[M]_E$ or simply by $[M]$.
\par
For a distinguished polynomial $f(S) \in \mathcal{O}_{E}[S] $,
we consider finitely generated torsion $\Lambda_E$-modules whose characteristic ideals are $(f(S))$,
and define the set $\mathcal{M}_{f(S)}^E$ by
\begin{equation*}
\mathcal{M}_{f(S)}^E=
\left \{
~[M]_E~
\vline
\begin{aligned}
&~M \textrm{~is~a~finitely~generated~torsion}~\Lambda_E \textrm{-module},\\
&~\mathrm{char}(M)=(f(S))\mathrm{~and~} M \mathrm{~is~free~over~} \mathcal{O}_E~
\end{aligned}
\right\}. \label{Mft}
\end{equation*}
Let $\overline{E}$ be a splitting field of $f(S)$.
Sumida and Koike considered the case of degree $2$.
In other words, 
there are some elements $\alpha$ and $\beta$ of $\overline{E}$ such that
\begin{eqnarray*}
f(S)=(S-\alpha )(S-\beta ).
\end{eqnarray*}
They classified all the elements of $\mathcal{M}_{f(S)}^E $ in \cite{Ko} and \cite{Su}.
Let us introduce their results in the following.
There are three cases to consider. 
$$
\begin{cases}
~~(\mathrm{i})~ &\textrm{The polynomial~} f(S) ~\textrm{is separable and reducible over } E.\\
~~(\mathrm{ii})~&\textrm{The polynomial~}f(S) ~\textrm{is irreducible over } E. \\
~~(\mathrm{iii})~ &\textrm{The polynomial~} f(S) ~\textrm{is inseparable}.
\end{cases}
$$
First, we consider the case of (i).
Let $f(S)$ be a separable and reducible polynomial.
In other words, we assume that
\begin{eqnarray*}
f(S)=(S-\alpha )(S-\beta ),
\end{eqnarray*}
where $\alpha$ and $\beta$ are distinct elements of $\pi_E \mathcal{O}_E$.
Let $[M]_E $ be an element of $\mathcal{M}_{f(S)}^E$.
Since $M$ has no non-trivial finite $\Lambda_E$-submodule,
there exists an injective $\Lambda_E$-homomorphism
\[
\varphi:M \hookrightarrow \Lambda_E /(S-\alpha ) \oplus \Lambda_E/(S-\beta )
\]
with finite cokernel.
We fix the notation to express such submodules in $\Lambda_E/(S-\alpha ) \oplus \Lambda_E/(S-\beta )$.
By using the canonical isomorphism 
$\Lambda_E/(S-\alpha) \cong \mathcal{O}_E  ~~(g(S)\mapsto  g(\alpha))$,
we define an isomorphism
\[
\iota :\mathcal{E}=\Lambda_E/ (S- \alpha ) \oplus \Lambda_E / (S-\beta ) \longrightarrow \mathcal{O}_E^{\oplus 2}
\]
by $(g_1(S),g_2(S)) \mapsto (g_1(\alpha),g_2(\beta))$. 
We identify $\mathcal{E}$ with $\mathcal{O}_E^{\oplus 2}$ via $\iota$.
Thus an element in $\mathcal{E}$ is expressed as $(a_1,a_2) \in \mathcal{O}_E^{\oplus 2 }$.
Since the rank of $M$ is equal to two, we can write $M$ of the form
\[
M=
\langle
(a,b),(c,d)
\rangle_{\mathcal{O}_E}~ \subset~ \Lambda_E/(S-\alpha ) \oplus \Lambda_E/(S-\beta ),
\]
where $\langle * \rangle _{\mathcal{O}_E}$ is the $\mathcal{O}_E$-submodule generated by $*$.
Furthermore, using this notation, we can express the action of $S$ by
\[
S(a,b)=(\alpha a,\beta b).
\]
\begin{Rem}
\rm{
The module 
$M =
\langle
(a,b),(c,d)
\rangle_{\mathcal{O}_E}
$
is, in fact,
an $\Lambda_{E}$-module
(see \cite[Lemma 2.1 (i)]{Ko}).
}
\end{Rem}
Then Sumida proved the following
\begin{prop}{\rm (Sumida {\cite[Proposition $10$]{Su}})} \label{Sumida re}
Let $f(S)$ be the same polynomial as above. Then we have
$$
\mathcal{M}_{f(S)}^E=\{ 
[M(k)]_E
~|~ 0 \leq k \leq \mathrm{ord}_E(\beta - \alpha )\},
$$
where
$$
M(k)=
\langle
(1,1), (0, \pi_E^k)
\rangle_{\mathcal{O}_E} \subset \Lambda_E/(S-\alpha ) \oplus \Lambda_E/(S-\beta ).
$$
Furthermore, we have
$$
M (k) \cong M (k') \Leftrightarrow k=k'.
$$
\end{prop}~\par
Next, we consider the case of (ii).
Let $f(S)$ be an irreducible polynomial.
We put
\begin{eqnarray*}
f(S)=S^2+c_1 S +c_0 \in \mathcal{O}_E[S].
\end{eqnarray*}
By the same method as in the case of $\mathrm{(}$i$\mathrm{)}$,
we identify $M$ with the submodule of finite index in $\Lambda_E /(f(S))$ via
an injective $\Lambda_E$-homomorphism
\[
\varphi:M \hookrightarrow \Lambda_E /(f(S)).
\]
Then 
we can write $M$ of the form
\[
M=
\langle
a S+b,c S+d
\rangle_{\mathcal{O}_E}~ \subset~ \Lambda_E / ( f(S) ),
\]
where $a,b,c,$ and $d$ are elements of $\mathcal{O}_E$.
Furthermore, using this notation, we can express the action of $S$ by
\[
S(a S +b ,c S +d )=((b-ac_1)S -a c_0, (d-cc_1)S-cc_0).
\] 
\begin{Rem}
\rm{
The module 
$
M=
\langle
a S+b , c S+d
\rangle_{\mathcal{O}_E}
$,
in fact,
is an ${\Lambda_E}$-module
(see \cite[Lemma 2.1 (ii)]{Ko}).
}
\end{Rem}
Then Koike proved the following
\begin{thm}{\rm (Koike \cite[Theorem 2.1]{Ko})}
Let $f(S)$ be the same polynomial as above.
Then we have
\begin{eqnarray*}
{\mathcal{M}}^E_{f(S)}=\left\{ [N_{x}]_E~\vline~N_{x}=\left\langle S+\frac{c_1}{2},
\pi_E^x \right\rangle_{\mathcal{O}_E} , 0 \leq  x \leq \frac{1}{2}\mathrm{ord}_E(c_1^2-4c_0) \right\}.
\end{eqnarray*}
\end{thm}~\par
Koike also determined all the elements of ${\mathcal{M}}^E_f(S)$ in the case of (iii) (\cite[Theorem 2.1]{Ko}).
\begin{Rem}
\rm{
We note that there is an isomorphism of $\Lambda_E$-modules
\begin{eqnarray*}
M(k) \cong N_{\mathrm{ord}_E(\beta -\alpha )-k} \subset \Lambda_{E}/(S-\alpha )(S-\beta )
\end{eqnarray*}
by \cite[Theorem 2.1]{Ko}.
}
\end{Rem}
To compute $k$ in Proposition \ref{Sumida re},
we introduce the notion of the higher Fitting ideals and state
relationships between $\Lambda_E$-modules and their higher Fitting ideals.
For a commutative ring $R$ and a finitely presented $R$-module $M$, we consider the following exact sequence
\[
R^m \stackrel f \to R^n \rightarrow M \rightarrow 0,
\]
where $m$ and $n$ are positive integers. For an integer $i \geq 0$ such that
$0 \leq i < n$, the $i$-th Fitting ideal of $M$ is defined to be the ideal of
$R$ generated by all $(n-i) \times (n-i)$ minors of the matrix corresponding to $f$.
We denote the $i$-th Fitting ideal of $M$ by $\mathrm{Fitt}_{i,R} (M)$.
This definition does not depend on the choice of the exact sequence above (see \cite{No}).
The following lemma says that the isomorphism class of a finitely
generated torsion $\Lambda_E$-module $M$ with $\mathrm{rank}_{\mathcal{O}_E}(M) = 2$ is determined by the Fitting
ideals $\mathrm{Fitt}_{0,\Lambda_E}(M)$ and $\mathrm{Fitt}_{1,\Lambda_E}(M)$.
\begin{lem}{\rm (Kurihara \cite[Lemma 9.1]{Ku})}\label{kuri lem}
Put $f(S)=(S-\alpha )(S-\beta ) \in \mathcal{O}_E[S]$. Let $[M] $ be an element of $\mathcal{M}_{f(S)}^E$.
Suppose that $\alpha$ and $\beta$ belong to $\mathcal{O}_E$. Then we have an exact sequence of $\Lambda$-modules
\[
0 \rightarrow \Lambda_E^2 \stackrel h \to \Lambda_E^2 \rightarrow M \rightarrow 0
\]
such that the matrix $A_{h}$ corresponding to the $\Lambda_E$-homomorphism $h$ is of the form
\begin{eqnarray*}
A_{h}=\left(
\begin{array}{cc}
S-\alpha & \pi^i_E \\
0 &S-\beta 
\end{array}
\right)
\end{eqnarray*}
for some $i$ with $0<i \leq \mathrm{ord}_E(\beta -\alpha )$.
Here if $\alpha = \beta$, $i= \infty$ is allowed.
Further, the isomorphism class of $M$ is determined by the value $i$.
\end{lem}
\begin{cor}\label{kuri lem rem}
We suppose that the same assumption in Lemma \ref{kuri lem} holds.
We also assume that $\alpha \neq \beta$.
Let $M$ be a $\Lambda_E$-module satisfying $[M] \in \mathcal{M}_{f(S)}^E$
and $[M] = [M(k)]$ for some non-negative integer $k$ 
with $0 \leq k \leq \mathrm{ord}_E(\beta - \alpha)$.
Then we have 
\[
\mathrm{Fitt}_{0,\Lambda_E}(M(k)) =((S-\alpha )(S-\beta )),~ \mathrm{Fitt}_{1,\Lambda_E}(M(k))=(S-\alpha, (\beta - \alpha)\pi_E^{-k}).
\]
\begin{proof}
We have
\begin{eqnarray*}
S(1,1) &=& (\alpha, \beta)\\
        &=& \alpha (1,1) +(\beta - \alpha ) \pi_E^{-k}(0,\pi_E^{k}),\\
S(0,\pi_E^{k}) &=& \beta (0, \pi_E^{k}).
\end{eqnarray*}
Thus we get the conclusion.
\end{proof}
\end{cor}
\section{examples of Theorem \ref{main thm of classification lambda=2}}
In this section, we give examples of Theorem \ref{main thm of classification lambda=2}.
We use the same notation as in the previous section
and suppose that the assumption in Theorem \ref{main thm of classification lambda=2} holds.\par
\subsection{Setting}

We put $\Lambda=\mathbb{Z}_p[[S]]$ and
$K=\mathbb{Q}(\sqrt{-d})$, where $d$ is a positive square-free integer.
For each $n\geq 0 $, we denote by $K_n^{{\rm c}}$ the intermediate field of the cyclotomic $\mathbb{Z}_p$-extension $K_\infty^{{\rm c}}$
such that $K_n^{{\rm c}}$ is the unique cyclic extension  over $K$ of degree $p^n$.
Let $A_{K_n^{{\rm c}}}$ be the $p$-Sylow subgroup of the ideal class group of $K_n^{{\rm c}}$.
Then, by class field theory, we have
$\displaystyle{X_{K_\infty^{{\rm c}}}
\cong
\underleftarrow{\lim}  A_{K_n^{{\rm c}}}}$, where the inverse limit is
taken with respect to the relative norms. 
As in \S \ref{Introduction}, 
$X_{K_\infty^{{\rm c}}}$ is a finitely generated torsion $\Lambda$-module
via an fixed isomorphism 
\begin{eqnarray}\label{isom ring}
\mathbb{Z}_p[[ \mathrm{Gal}(K_\infty^{{\rm c}}/K)]]  \cong \mathbb{Z}_p[[S]] \qquad  (\sigma \leftrightarrow 1+S),
\end{eqnarray}
where $\sigma$ is a topological generator of $\mathrm{Gal}(K_\infty^{{\rm c}}/K)$.
Let $f(S)$ be the distinguished polynomial which generates char$(X_{K_\infty^{{\rm c}}})$.
Since it is known that $X_{K_\infty^{{\rm c}}}$ is a free $\mathbb{Z}_p$-module,
we have $[X_{K_\infty^{{\rm c}}}]_{\mathbb{Q}_p} \in \mathcal{M}_{f(S)}^{\mathbb{Q}_p}$.
We can calculate the polynomial $f(S)$ mod $p^n$ for small $n$ numerically.
We can compute $f(S)$ by PARI/GP 
\cite{PARI/GP}
or  by Mizusawa's program Iwapoly.ub
\cite[Research, Programing, Approximate Computation of Iwasawa Polynomials by UBASIC]{Mi}.\par
Let $E$ be the minimal splitting field of $f(S)$.
Note that $f(S)$ is separable by the assumption in Theorem \ref{main thm of classification lambda=2}.
As in the previous section, there exist an integer $k$ with $0 \leq  k \leq  {\mathrm{ord}}_{E} (\beta-\alpha)$,
which depends only on the isomorphism class of $X_{K_\infty^{{\rm c}}}$,
and an $\mathcal{O}_E$-basis $\mathbf{e}_1, \mathbf{e}_2$ of $X_{K_\infty^{{\rm c}}}\otimes_{\mathbb{Z}_p} \mathcal{O}_E$
such that the homomorphism on $\Lambda_E$-modules
\begin{eqnarray}\label{definition of k}
X_{K_\infty^{{\rm c}}}\otimes_{\mathbb{Z}_p} \mathcal{O}_E
\hookrightarrow
\Lambda_E/(S-\alpha) \oplus \Lambda_E/(S-\beta);
\mathbf{e}_1
\mapsto
(1,1),
\mathbf{e}_2
\mapsto
(0,\pi_E^{k})
\end{eqnarray}
is injective.
In the case of $k=0$, we have
$X_{K_\infty^{{\rm c}}}\otimes_{\mathbb{Z}_p} \mathcal{O}_E \cong \Lambda_E/(S-\alpha) \oplus \Lambda_E/(S-\beta)$.
In this case we use the standard basis $\{ (1,0), (0,1)\}$ instead of $\{ (1,1), (0,1)\}$.
We regard $X_{K_\infty^{{\rm c}}}\otimes_{\mathbb{Z}_p} \mathcal{O}_E$ 
as a $\Lambda_E$-submodule of $\Lambda_E/(S-\alpha) \oplus \Lambda_E/(S-\beta)$ by the above injection.
For convenience, we regard $X_{K_\infty^{{\rm c}}} \subset X_{K_\infty^{{\rm c}}}\otimes_{\mathbb{Z}_p} \mathcal{O}_E$
by the injection $x \mapsto  x \otimes 1$.
We can take generators $x_1$ and $x_2$ of $X_{K_\infty^{{\rm c}}}$ satisfying the following condition (CG)
(see Section $4$ in \cite{MOMO}):\\[3pt]~
{\bf condition of generators (CG)}
\begin{itemize}\label{condition of generators}
\item
$x_1$ and $x_2$ generate $X_{K_\infty^{{\rm c}}}$.
\item
The image of $x_1$ generates $\mathrm{Gal}(L_K \cap \widetilde{K}/K)$.
The image of $x_2$
becomes $0$ in $\mathrm{Gal}(L_K \cap \widetilde{K}/K)$.
\end{itemize}
We denote by $(\mu_{11}, \mu_{12})$ (resp. $(\mu_{21}, \mu_{22})$) the image of $x_1 \otimes 1$ (resp. $x_2 \otimes 1$)
under the map (\ref{definition of k}).
Then we can write
\begin{eqnarray*}
x_1 &=& \lambda_{11} \mathbf{e}_1 + \lambda_{12} \mathbf{e}_2=(\mu_{11},\mu_{12}),  \\
x_2 &=& \lambda_{21} \mathbf{e}_1 + \lambda_{22} \mathbf{e}_2=(\mu_{21},\mu_{22})
\end{eqnarray*}
for some $\lambda_{ij} \in \mathcal{O}_E$.
Note that $\lambda_{21}=\mu_{21}$ and that
$\lambda_{11}\lambda_{22}-\lambda_{21}\lambda_{22} \in \mathcal{O}_E^{\times}$.
Moreover, if $k=0$, then $\lambda_{ij}=\mu_{ij}$ $(i,j=1,2)$
since we take $\mathbf{e}_1=(1,0), \mathbf{e}_2=(0,1)$.
We will use the following lemma in the next subsection.
\begin{lem}\label{main lem}
With the same notation as above, we have the following:\\
$\mathrm{(i)}$ If $k=0$, we have
\begin{eqnarray*}
Sx_1 &=&
\frac{\alpha \lambda_{11}\lambda_{22} - \beta \lambda_{12}\lambda_{21}}{\det(\lambda_{ij})_{ij}}
x_1 +
\frac{(\beta - \alpha )\lambda_{11}\lambda_{12}}{\det(\lambda_{ij})_{ij}}
x_2,
\\
Sx_2 &=&
\frac{(\alpha-\beta) \lambda_{21}\lambda_{22}}{\det(\lambda_{ij})_{ij}}
x_1 +
\frac{-\alpha \lambda_{12}\lambda_{21}+\beta \lambda_{11}\lambda_{22} }{\det(\lambda_{ij})_{ij}}
x_2.
\end{eqnarray*}
$\mathrm{(ii)}$ If $k>0$, we have
\begin{eqnarray*}
Sx_1 &=&
\frac{\alpha \lambda_{11}\lambda_{22} - \beta \lambda_{12}\lambda_{21} -\lambda_{11}\lambda_{21} \gamma}{\det(\lambda_{ij})_{ij}}
x_1+
\frac{(\beta-\alpha) \lambda_{11} \lambda_{12} +\lambda_{11}^2 \gamma}{\det(\lambda_{ij})_{ij}}
x_2,
\\
Sx_2 &=&
\frac{(\alpha-\beta)\lambda_{21} \lambda_{22} -\lambda_{21}^2 \gamma}{\det(\lambda_{ij})_{ij}}
x_1+
\frac{-\alpha \lambda_{12} \lambda_{21} + \lambda_{11}\lambda_{22}\beta +\lambda_{11}\lambda_{21} \gamma }{\det(\lambda_{ij})_{ij}}
x_2,
\end{eqnarray*}
where $\gamma = (\alpha-\beta) \pi_E^{-k}$.
\end{lem}
\begin{proof}
Put $\delta=0$ or $1$ according to whether $k=0$ or not.
Then we have
\begin{eqnarray*}
\begin{pmatrix}
x_1
\\
x_2
\end{pmatrix}
=
\begin{pmatrix}
\lambda_{11} & \lambda_{12}
\\
\lambda_{21} & \lambda_{22}
\end{pmatrix}
\begin{pmatrix}
\mathbf{e}_1
\\
\mathbf{e}_2
\end{pmatrix}
,\ \ \ \ 
S
\begin{pmatrix}
\mathbf{e}_1
\\
\mathbf{e}_2
\end{pmatrix}
=
\begin{pmatrix}
\alpha & \gamma\delta
\\
0 & \beta 
\end{pmatrix}
\begin{pmatrix}
\mathbf{e}_1
\\
\mathbf{e}_2
\end{pmatrix}.
\end{eqnarray*}
Therefore we have
\begin{eqnarray*}
S
\begin{pmatrix}
x_1
\\
x_2
\end{pmatrix}
=
\begin{pmatrix}
\lambda_{11} & \lambda_{12}
\\
\lambda_{21} & \lambda_{22}
\end{pmatrix}
\begin{pmatrix}
\alpha & \gamma\delta
\\
0 & \beta 
\end{pmatrix}
\begin{pmatrix}
\lambda_{11} & \lambda_{12}
\\
\lambda_{21} & \lambda_{22}
\end{pmatrix}^{-1}
\begin{pmatrix}
x_1
\\
x_2
\end{pmatrix}.
\end{eqnarray*}
We obtain the results from this equation.
\end{proof}
\subsection{A method of computing $\mu_{21}$ and $\mu_{22}$}\label{A method of computing mu_{21} and mu_{22}}

In this subsection, we give a method of computing $\mu_{21}$ and $\mu_{22}$
in Theorem \ref{main thm of classification lambda=2}.
We assume that
$A_{K} \cong \mathbb{Z}/p^{n_1}\mathbb{Z} \oplus \mathbb{Z}/p^{n_2}\mathbb{Z}$
for some non-negative integers $n_1, n_2$.
Then we have
$A_K \otimes_{\mathbb{Z}_p} \mathcal{O}_E \cong \mathcal{O}_E/\pi_E^{N_1}\mathcal{O}_E \oplus \mathcal{O}_E/\pi_E^{N_2}\mathcal{O}_E$,
where $N_i=en_i~(i=1,2)$ and $e$ is the ramification index in $E/\mathbb{Q}_p$ for $p$.
We also assume that we have a direct sum decomposition
\begin{eqnarray}\label{direct sum decomposition}
\mathrm{Gal}(L_K/K) \cong \mathrm{Gal}(L_K \cap \widetilde{K}/K) \oplus \mathrm{Gal}(L_K/L_K \cap \widetilde{K})
\end{eqnarray}
with the order of $\mathrm{Gal}(L_K \cap \widetilde{K}/K)$ is $p^{n_1}$.
Since $p$ does not split in $K$, we have $\Lambda$-isomorphisms
\begin{equation*}
\psi_{n} : X_{K_\infty^{{\rm c}}}/\omega_n(S) X_{K_\infty^{{\rm c}}} \stackrel\sim\to A_{K_n^{{\rm c}}}  \label{ideal}
\end{equation*}
for any non-negative integers $n$,
where $\omega_n(S)=(1+S)^{p^n}-1$ (see \cite[Proposition $13.22$]{Wa}).\par
Recall that the Iwasawa $\lambda$-invariant of $K_\infty^{{\rm c}}/K$ is $2$.
Hence $K_n^{{\rm c}}$ is generated by two elements.
Since the order of $\mathrm{Gal}(L_K \cap \widetilde{K}/K)$ is $p^{n_1}$, we have $L_K \cap \widetilde{K} = K_{n_1}^{\rm an}$,
where $K_{n_1}^{\rm an}$ is the $n_1$-th layer of the anti-cyclotomic $\mathbb{Z}_p$-extension $K_\infty^{\rm an}/K$.
Fix a non-negative integer ${n}$.
We take a basis $\{[\frak{b}_1], [\frak{b}_2] \}$ of $A_{K^{\rm c}_{{n}}}$ satisfying the following:\par
\begin{itemize}
\item[(i)]
$\displaystyle{[\frak{b}_1]=
s[\frak{Q}_1],
[\frak{b}_2]=
t[\frak{L}_1}]
$
for some non-negative integers $s$, $t$ and for some prime ideals $\frak{Q}_1, \frak{L}_1$.
Here we denote by $[*]$ the ideal class of $*$.\par
\item[(ii)]
$\frak{Q}_1, \frak{L}_1$ are prime ideals in $K_{{n}}^{{\rm c}}$ lying above primes $q, \ell$, respectively.\par
\item[(iii)]
$q$ and $\ell$ split completely in $K_{{n}}^{{\rm c}}/\mathbb{Q}$, respectively.
\end{itemize}
Let $\frak{q},\overline{\frak{q}}, \frak{l}$, and $\overline{\frak{l}}$ be prime ideals in $K$
such that $q\mathcal{O}_{K}=\frak{q}\overline{\frak{q}}$
and $\ell\mathcal{O}_{K}=\frak{l}\overline{\frak{l}}$.
We write
\begin{eqnarray*}
\begin{array}{llll}
q \mathcal{O}_{K_{{n}}^{{\rm c}}} = \frak{Q}_{1} \overline{\frak{Q}_{1}} \cdots \frak{Q}_{p^{{n}}} \overline{\frak{Q}_{p^{{n}}}},
&
\frak{Q}_i ~|~\frak{q},
&
\overline{\frak{Q}_i} ~|~ \overline{\frak{q}}
&
(1=1,\ldots,p^n),
\\[3pt]
\ell \mathcal{O}_{K_{{n}}^{{\rm c}}} = \frak{L}_{1} \overline{\frak{L}_{1}} \cdots \frak{L}_{p^{{n}}} \overline{\frak{L}_{p^{{n}}}}, 
&
\frak{L}_i ~|~\frak{l},
&
\overline{\frak{L}_i} ~|~ \overline{\frak{l}}
&
(1=1,\ldots,p^n),
\end{array}
\end{eqnarray*}
where $\frak{Q}_{i}, \overline{\frak{Q}_i}, \frak{L}_i,$ and $\overline{\frak{L}_{i}}$ are prime ideals in $\mathcal{O}_{K_{{n}}^{{\rm c}}}$.
Since the norm map
$N_{K_{{n}}^{{\rm c}}/K} : A_{K_{{n}}^{{\rm c}}} \rightarrow A_K$ is surjective, we have
\begin{eqnarray}\label{L_K/K}
\mathrm{Gal}(L_K / K)
= \left\langle \left( \frac{L_K/K}{\frak{q}} \right)^s,  \left( \frac{L_K/K}{\frak{l}}\right)^t \right\rangle,
\end{eqnarray}
where
$\displaystyle{\left( \frac{L_K/K}{\frak{q}} \right)}$,
$\displaystyle{\left( \frac{L_K/K}{\frak{l}} \right)}$ are the Frobenius endomorphism of $\frak{q}, \frak{l}$, respectively.
By our assumption (\ref{direct sum decomposition}), there exist non-negative integers $u,v$ such that $s \mid u$, $t \mid v$ and
\begin{eqnarray}\label{L_K/cap}
\mathrm{Gal}(L_K / L_K \cap \widetilde{K})
= \mathrm{Gal}(L_K / K_{n_1}^{{\rm an}})
= 
\left\langle \left( \frac{L_K/K}{\frak{q}} \right)^{u} \!\! \left( \frac{L_K/K}{\frak{l}}\right)^{v} \right\rangle.
\end{eqnarray}
Let $Q$ and $L$ be the fields corresponding to the subgroups generated by
$\displaystyle{\left( \frac{L_K/K}{\frak{q}} \right)^s}$, $\displaystyle{\left( \frac{L_K/K}{\frak{l}} \right)^t}$,
respectively.
Then we have $L_K=QK_{n_1}^{{\rm an}}$ or $L_K=LK_{n_1}^{{\rm an}}$.
We may assume that $L_K=QK_{n_1}^{{\rm an}}$.
Furthermore, we have the following commutative diagram:
$$
\begin{CD}
  X_{K_\infty^{{\rm c}}}   \\ 
 @V VV \\  
  X_{K_\infty^{{\rm c}}}/\omega_{{n}}(S) X_{K_\infty^{{\rm c}}}   @>>\psi_{{n}} > A_{K_{{n}}^{{\rm c}}} \\
 @V VV @VV{N_{K_{{n}}^{{\rm c}}/K}}V\\
  X_{K_\infty^{{\rm c}}}/S X_{K_\infty^{{\rm c}}}   @>>\psi_{0}>  A_{K}.
 \end{CD}
$$ 
Using the commutative diagram above, we obtain $x_1, x_2 \in X_{K_\infty^{{\rm c}}}$
such that
\[
\psi_{{n}}(x_1 \mathrm{~mod~} \omega_{{n}}(S)) = [s\frak{Q}_{1}],\quad
\psi_{{n}}(x_2 \mathrm{~mod~} \omega_{{n}}(S)) = [u\frak{Q}_1+v\frak{L}_1],
\]
where $[s\frak{Q}_{1}], [u\frak{Q}_1+v\frak{L}_1]$ are the ideal classes of
$s\frak{Q}_{1}$ and $u\frak{Q}_1+v\frak{L}_1$, respectively.
By Nakayama's lemma and our assumptions,
we obtain $X_{K_\infty^{{\rm c}}}=\langle x_1,x_2\rangle$ and
$A_{K_{{n}}^{{\rm c}}} = \langle [s\frak{Q}_1], [u\frak{Q}_1+v\frak{L}_1] \rangle$.
These $x_1$ and $x_2$ satisfy the condition (CG).\par
Because $\mathbb{Z}_p[\mathrm{Gal}(K_{{n}}^{{\rm c}}/K)] \cong \Lambda/\omega_{{n}}(S) \Lambda,$ we get
\begin{eqnarray}\label{overline{S} ([ufrak{Q}_1+vfrak{L}_1])}
\overline{S} ([u\frak{Q}_1+v\frak{L}_1]) = A[s\frak{Q}_{1}] + B[u\frak{Q}_1+v\frak{L}_1]
\end{eqnarray}
for some $A,B \in \mathbb{Z}_p$, where $\overline{S} = S\mathrm{~mod~}\omega_{{n}}(S)$.
Then we obtain the following theorem, which gives a method
of computing $\mu_{21}$ and $\mu_{22}$.
\begin{thm}\label{main thm}
\makeatletter
  \parsep   = 0pt
  \labelsep = .5pt
  \def\@listi{%
     \leftmargin = 20pt \rightmargin = 0pt
     \labelwidth\leftmargin \advance\labelwidth-\labelsep
     \topsep     = 0\baselineskip
     \partopsep  = 0pt \itemsep       = 0pt
     \itemindent = 0pt \listparindent = 10pt}
  \let\@listI\@listi
  \@listi
  \def\@listii{%
     \leftmargin = 20pt \rightmargin = 0pt
     \labelwidth\leftmargin \advance\labelwidth-\labelsep
     \topsep     = 0pt \partopsep     = 0pt \itemsep   = 0pt
     \itemindent = 0pt \listparindent = 10pt}
  \let\@listiii\@listii
  \let\@listiv\@listii
  \let\@listv\@listii
  \let\@listvi\@listii
  \makeatother
With the same notation as above, 
$\mathrm{ord}([s\frak{Q}_{1}])$ denotes the order of $[s\frak{Q}_{1}]$
in the ideal class group $A_{K_{{n}}^{{\rm c}}}$.
Then we have the following:
\begin{itemize}
\item[$(\mathrm{a})$] 
If $k=0$ 
and $\mathrm{ord}_E(A)={\mathrm{ord}_E(\beta-\alpha)}<\mathrm{ord}([s\frak{Q}_{1}])$,
then
$\mathrm{ord}_E(\mu_{21})=\mathrm{ord}_E(\mu_{22})=0$.
\item[$(\mathrm{b})$] 
If $k>0$ and 
$\mathrm{ord}_E(\beta - \alpha)-k=\mathrm{ord}_E(A) <\mathrm{ord}([s\frak{Q}_{1}])$, then $\mathrm{ord}_E(\mu_{21})=0$.
\end{itemize}
\end{thm}
\begin{proof}
(a) Note that $\lambda_{11} \lambda_{22}-\lambda_{12} \lambda_{21} \in \mathcal{O}_E^{\times}.$
If $k=0$, then by Lemma \ref{main lem}, we obtain
\[
Sx_2 =
\frac{(\alpha-\beta) \lambda_{21}\lambda_{22}}{\det(\lambda_{ij})_{ij}}
x_1 +
\frac{-\alpha \lambda_{12}\lambda_{21}+\beta \lambda_{11}\lambda_{22} }{\det(\lambda_{ij})_{ij}}
x_2.
\]
Comparing (\ref{overline{S} ([ufrak{Q}_1+vfrak{L}_1])}) with this,
we obtain
\begin{eqnarray*}
{\rm ord}_E(A) 
\equiv
{\rm ord}_E((\alpha-\beta) \lambda_{21}\lambda_{22})
\ \ \ \ 
\text{mod $\mathrm{ord}([s\frak{Q}_{1}])$}.
\end{eqnarray*}
The claim follows from this and $\lambda_{ij}=\mu_{ij}$ $(i,j=1,2)$.
\\
(b) 
Similarly, if $k>0$, then by Lemma \ref{main lem} and (\ref{overline{S} ([ufrak{Q}_1+vfrak{L}_1])}),
we obtain
\begin{eqnarray*}
{\rm ord}_E(A) 
\equiv
{\rm ord}_E(\lambda_{21})+{\rm ord}_E((\alpha-\beta)  \lambda_{22} -\lambda_{21} (\beta-\alpha)\pi_E^{-k})
\ \ \ 
\text{mod $\mathrm{ord}([s\frak{Q}_{1}])$}.
\end{eqnarray*}
The right hand side of this congruence becomes
\begin{eqnarray*}
&&
{\rm ord}_E(\lambda_{21})+{\rm ord}_E((\alpha-\beta)  \lambda_{22} -\lambda_{21} (\beta-\alpha)\pi_E^{-k})
\\
&=&
{\rm ord}_E(\lambda_{21})+{\rm ord}_E(\alpha-\beta)-k+{\rm ord}_E(\lambda_{21}+\lambda_{22}\pi_E^k)
\\
& \ge &
{\rm ord}_E(\lambda_{21})+{\rm ord}_E(\alpha-\beta)-k.
\end{eqnarray*}
The claim follows from this and $\lambda_{21}=\mu_{21}$.
\end{proof}
\begin{Rem}
\rm
We note that it is enough to calculate Galois actions on ideal classes
by the method above 
in the case of ${n}=\mathrm{ord}_E(\beta - \alpha)$.
Indeed, we have
\[
A_{K_n^{{\rm c}}} \cong \mathbb{Z}/p^{n_1 + n}\mathbb{Z} \oplus \mathbb{Z}/p^{n_2 + n}\mathbb{Z}
\]
for $n \geq 0$ by \cite[Proposition $2.2$]{Ko}.
Hence, if ${n}=\mathrm{ord}_E(\beta - \alpha)$, then $\mathrm{ord}([s\frak{Q}_{1}])=n_1+n>\mathrm{ord}_E(\beta - \alpha)$.
\end{Rem}
\subsection{Examples of Theorem \ref{main thm of classification lambda=2}}
\begin{Example}\label{example of 1.2(i)}
\begin{rm}
Let $p=3$ and $K=\mathbb{Q}(\sqrt{-12394})$.
Using PARI/GP, we have $A_K \cong \mathbb{Z}/9 \mathbb{Z} \oplus \mathbb{Z}/3\mathbb{Z}$.
By Lemma \ref{Fujii}, we have $L_{K} \cap \widetilde{K}=K_{2}^{{\rm an}}$.
Indeed, we have
$(I(3)/S(3^5)) \otimes \mathbb{Z}_3 \cong \mathbb{Z}/3\mathbb{Z}
\oplus \mathbb{Z}/3^{4}\mathbb{Z} \oplus \mathbb{Z}/3^{6} \mathbb{Z}.$
Hence we  get
$\mathrm{Gal}(L_K/L_K \cap \widetilde{K}) \cong \mathbb{Z}/3 \mathbb{Z}$.
This implies that  $L_{K} \cap \widetilde{K}=K_{2}^{{\rm an}}$.
Moreover, using \cite[Theorem $2$]{Br}, we obtain
\begin{eqnarray*}\label{18th}
S^{18} + 18S^{16} + 1069S^{14} - 4372S^{12} + 152180S^{10} - 1347136S^{8} +\\ \nonumber
2053184S^{6} + 36414976S^{4} - 166023168S^{2} + 203063296
\end{eqnarray*}
as a defining polynomial of $K_{2}^{{\rm an}}$ over $\mathbb{Q}$.
By PARI/GP, we have 
\[
f(S)\equiv 
S^2 + 90 S + 189~~\mathrm{mod}~3^5.
\]
Let $E$ be the minimal splitting field of $f(S)$.
We put $f(S)=(S-\alpha)(S-\beta)$, where $\alpha$ and $\beta \in E$.
We can check that $E/\mathbb{Q}_p$ is an ramified extension and we get
$\mathrm{ord}_E(\alpha-\beta)=3.$
By the table in \cite{Ko}, we obtain
\[
X_{K_\infty^{{\rm c}}} \otimes_{\mathbb{Z}_p} \mathcal{O}_E \cong \langle 
(1,1), (0, \pi_{E}^2) \rangle_{\mathcal{O}_E},
\]
which implies that $k=2$ in Theorem \ref{main thm of classification lambda=2}.
Since we have $\mathrm{ord}_E(\alpha)=\mathrm{ord}_E(\beta)=3$,
we obtain $\mathrm{ord}_E(\alpha-\beta)-k=1<3$.
Therefore $X_{\widetilde{K}}$ is cyclic as $\mathbb{Z}_p[[\mathrm{Gal}(\widetilde{K}/K)]]$-modules
by Theorem \ref{main thm of classification lambda=2} (i).
\end{rm}
\end{Example}
We can also obtain the same result as above by the following 
\begin{prop}\label{test}
We use the same notation as above.
Suppose the following conditions:\par
$\mathrm{(i)}$ $\mathrm{ord}_E(\alpha)=\mathrm{ord}_E(\beta),$\par
$\mathrm{(ii)}$ $A_K \cong \mathbb{Z}/p^{m_1} \mathbb{Z} \oplus \mathbb{Z}/p^{m_2} \mathbb{Z}$ ~~$(m_1 < m_2)$,\par
$\mathrm{(iii)}$ $L_{K} \cap \widetilde{K}=K_{m_2}^{{\rm an}}$.\\
Then $X_{\widetilde{K}}$ is cyclic as $\mathbb{Z}_p[[\mathrm{Gal}(\widetilde{K}/K)]]$-modules.
\end{prop}
\begin{proof}
Using \cite[Lemma 5.2]{MOMO}, we have
\begin{eqnarray*}
A_K \otimes_{\mathbb{Z}_p} \mathcal{O}_E \cong
\begin{cases}
\mathcal{O}_E/\alpha \mathcal{O}_E \oplus \mathcal{O}_E/\beta \mathcal{O}_E
&\mathrm{~if~} \mathrm{ord}_E(\beta - \alpha) -k \geq m, \\
\mathcal{O}_E/(\beta - \alpha)\pi_E^{-k} \mathcal{O}_E \oplus \mathcal{O}_E/\frac{\alpha \beta}{(\beta - \alpha)\pi_E^{-k}} \mathcal{O}_E.
&\mathrm{~if~} \mathrm{ord}_E(\beta - \alpha) -k < m,
\end{cases}
\end{eqnarray*}
where $m={\rm min}\{{\rm ord}_E(\alpha), {\rm ord}_E(\beta)\}$.
This implies that $k>0$ by assumptions (i) and (ii).
Hence we have $\mathrm{ord}_E(\beta - \alpha) -k < m$.
Moreover, ${\rm Gal}(L_K \cap \widetilde{K}/K)$ is a direct summand of ${\rm Gal}(L_K/K)$ by (iii).
By (i) in Theorem \ref{main thm of classification lambda=2}, we get the conclusion.
\end{proof}
By Proposition \ref{test}, we obtain the Tables 1 and 2.
\begin{table}[htb]
\caption{}
\begin{tabular}{|c||c|c|c|c|c|c|c|c|} \hline
$d$&
\scriptsize{$\mathrm{ord}_E(\alpha-\beta)$}&
$k$&
$m$&
\scriptsize{$L_{K} \ \cap \widetilde{K}$}&
$E/\mathbb{Q}_3$&
$A_{0}$&
$X_{\widetilde{K}}$\\
\hline \hline
5703&
3&
2&
3&
$K_{2}^{{\rm an}}$&
\footnotesize{ramified}&
$(9,3)$&
cyclic\\
12394&
3&
2&
3&
$K_{2}^{{\rm an}}$&
\footnotesize{ramified}&
$(9,3)$&
cyclic\\
50293&
3&
2&
3&
$K_{2}^{{\rm an}}$&
\footnotesize{ramified}&
$(9,3)$&
cyclic\\
54931&
3&
2&
3&
$K_{2}^{{\rm an}}$&
\footnotesize{ramified}&
$(9,3)$&
cyclic\\
89269&
3&
2&
2&
${K_3}^{{ \rm an}}$&
\footnotesize{unramified}&
$(27,3)$&
cyclic\\
\hline
\end{tabular}
\\[2pt]
(The integer $k$ is defined by (\ref{definition of k}) and $m={\rm min}\{{\rm ord}_E(\alpha), {\rm ord}_E(\beta)\}$)
\end{table}
\begin{table}[htb]
\caption{}
\begin{tabular}{|c||c|} \hline
$d$&
a generator of $ \mathrm{char}(X_{K_\infty^{{\rm c}}})~\mathrm{~mod~}3^5$ \\
\hline \hline
5703&
$S^2 + 63S + 135$\\
12394&
$S^2 + 63S + 27$\\
50293&
$S^2 + 54S + 189$\\
54931&
$S^2 + 135S + 216$\\
89269&
$S^2 + 63S + 81$\\
\hline
\end{tabular}
\end{table}

\vspace{10pt}

Next, using (iv) in Theorem \ref{main thm of classification lambda=2},
we obtain the following example that $X_{\widetilde{K}}$ is not cyclic as $\mathbb{Z}_p[[\mathrm{Gal}(\widetilde{K}/K)]]$-modules.
\begin{Example}\label{example of 1.2(iii)(iv)}
\begin{rm}
Let $p=3$ and $K=\mathbb{Q}(\sqrt{-42619})$.
Using PARI/GP, we have $A_K \cong \mathbb{Z}/3\mathbb{Z} \oplus \mathbb{Z}/3\mathbb{Z}$.
We have $L_{K} \cap \widetilde{K}=K_{1}^{{\rm an}}$.
Hence  ${\rm Gal}(L_K \cap \widetilde{K}/K)$ is a direct summand of ${\rm Gal}(L_K/K)$.
We get 
\begin{eqnarray*}
f(S)  \equiv 
S^2 + 186 S + 630~~\mathrm{mod}~3^6.
\end{eqnarray*}
By Hensel's Lemma, there exist $\alpha, \beta \in \mathbb{Z}_p$ such that
$f(S)=(S-\alpha)(S-\beta)$, where $\alpha \equiv 105~~\mathrm{mod}~3^5$ and $\beta \equiv 51~~\mathrm{mod}~3^5$.
Hence we have $\mathrm{ord}_p(\alpha-\beta)=3.$
In this case, although \cite{Ko} could not determine the isomorphism class of $X_{K_\infty^{{\rm c}}}$,
we can determine it using Fitting ideals as follows.
We compute
\[
A_{K_1^{{\rm c}}} = \mathbb{Z}/9  \mathbb{Z}
~ [\frak{b}_1]
\oplus \mathbb{Z}/9 \mathbb{Z}
~ [\frak{b}_2]
\]
for some ideals $\frak{b}_1$ and $\frak{b}_2$ in $\mathcal{O}_{K_1^{{\rm c}} }$.
Take a generator $\overline{\sigma}$ of $\mathrm{Gal}(K_1^{{\rm c}}/K)$.
These $\frak{b}_1$, $\frak{b}_2$, and $\overline{\sigma}$ are computed by PARI/GP.
We do not write down $\overline{\sigma}$ because it is complicated.
There is a topological generator ${\sigma} \in \mathrm{Gal}(K_\infty^{{\rm c}}/K)$
such that ${\sigma}$ is an extension of $\overline{\sigma}$.
By this topological generator, we have the isomorphism (\ref{isom ring}).
We regard $X_{K_\infty^{{\rm c}}}$ as a $\mathbb{Z}_p[[S]]$-module by this isomorphism.
We note that
$f(S)$ depends on the choice of ${\sigma}$, but we can easily
check that $\mathrm{ord}_p(\alpha)$, $\mathrm{ord}_p(\beta)$, and $\mathcal{M}_{f(S)}^{\mathbb{Q}_p}$ do not depend on the choice
of ${\sigma}$.
We also compute that
\[
\overline{\sigma} [\frak{b}_1] = 4 [\frak{b}_1],\quad
\overline{\sigma} [\frak{b}_2] = 4[\frak{b}_2].
\]
Hence we have
\[
\mathrm{Fitt}_{1,\mathbb{Z}_p[[S]]}(X_{K_\infty^{{\rm c}}}/ \omega_1(S) X_{K_\infty^{{\rm c}}} )=(S-3).
\]
Using Lemma \ref{kuri lem} and 
Corollary \ref{kuri lem rem},
we obtain $k=0$ in Theorem \ref{main thm of classification lambda=2}, which implies that
\begin{eqnarray*}
X_{K_\infty^{{\rm c}}}  &\cong & \langle (1,0), (0, 1) \rangle\\
                                     &=     & \Lambda/(S-\alpha ) \oplus \Lambda/(S-\beta ).
\end{eqnarray*}
Furthermore, we have $\mathrm{ord}_p(\alpha-\beta)-k=3> \mathrm{min} \{\mathrm{ord}_p(\alpha), \mathrm{ord}_p(\beta) \}=1$.
Therefore $X_{\widetilde{K}}$ is not cyclic as $\mathbb{Z}_p[[\mathrm{Gal}(\widetilde{K}/K)]]$-modules
by Theorem \ref{main thm of classification lambda=2} (iv).
\end{rm}
\end{Example}
By the same methods as in Examples \ref{example of 1.2(iii)(iv)} for $p=3$, we obtain Tables 3 and 4.
\begin{table}[htb]
\caption{}
\begin{tabular}{|c||c|c|c|c|c|c|c|c|} \hline
$d$&
\scriptsize{$\mathrm{ord}_E(\alpha-\beta)$}&
$k$&
$m$&
\scriptsize{$L_{K} \ \cap \widetilde{K}$}&
$E/\mathbb{Q}_3$&
$A_{0}$&
$X_{\widetilde{K}}$\\
\hline \hline
32137&
2&
0&
1&
$K_{1}^{{\rm an}}$&
\footnotesize{unramified}&
$(3,3)$&
non-cyclic\\
34989&
5&
1&
2&
$K_{1}^{{\rm an}}$&
\footnotesize{ramified}&
$(3,3)$&
non-cyclic\\
42619&
3&
0&
1&
$K_{1}^{{\rm an}}$&
\footnotesize{$E=\mathbb{Q}_p$}&
$(3,3)$&
non-cyclic\\
\hline
\end{tabular}
\\[2pt]
(The integer $k$ is defined by (\ref{definition of k}) and $m={\rm min}\{{\rm ord}_E(\alpha), {\rm ord}_E(\beta)\}$)
\end{table}
\begin{table}[htb]
\caption{}
\begin{tabular}{|c||c|} \hline
$d$&
a generator of $ \mathrm{char}(X_{K_\infty^{{\rm c}}})~\mathrm{~mod~}3^6$ \\
\hline \hline
32137&
$S^2 + 318S + 657$\\
34989&
$S^2 + 66 S + 117$ \\
42619&
$S^2 + 573S + 252$\\
\hline
\end{tabular}
\end{table}

\vspace{10pt}

On the other hand, using (iv) in Theorem \ref{main thm of classification lambda=2},
we obtain the following example that $X_{\widetilde{K}}$ is cyclic as $\mathbb{Z}_p[[\mathrm{Gal}(\widetilde{K}/K)]]$-modules.
\begin{Example}\label{example of 1.2(iv)}
\begin{rm}
Let $p=3$ and $K= \mathbb{Q}(\sqrt{-2437})$.
We will prove that $X_{\widetilde{K}}$ is a $\mathbb{Z}_{p}[[\mathrm{Gal}(\widetilde{K}/K)]]$-cyclic module using PARI/GP  \cite{PARI/GP}.
In this case we have
$\mathrm{Cl}_{K} \cong \mathbb{Z}/6\mathbb{Z} \oplus \mathbb{Z}/3 \mathbb{Z}$
and 
$\mathrm{Cl}_{K_1^{{\rm c}}} \cong \mathbb{Z}/3906 \mathbb{Z} \oplus \mathbb{Z}/9 \mathbb{Z}$.
Hence we have
$A_{K} \cong \mathbb{Z}/3\mathbb{Z} \oplus \mathbb{Z}/3 \mathbb{Z}$
and
$A_{K_1^{{\rm c}}} \cong \mathbb{Z}/9\mathbb{Z} \oplus \mathbb{Z}/9 \mathbb{Z}$.
We have 
\[
f(S)\equiv S^2+ 9 S+ 9~~\mathrm{mod}~3^3.
\]
Let $E$ be the minimal splitting field of $f(S)$.
We put $f(S)=(S-\alpha)(S-\beta)$, where $\alpha$ and $\beta \in E$.
Since the discriminant of $f(S)$ is $45$ mod $3^3$,
$E/\mathbb{Q}_p$ is an unramified extension and we get
$\mathrm{ord}_E(\alpha-\beta)=1.$
By the table in \cite{Ko}, we obtain
\[
X_{K_\infty^{{\rm c}}} \otimes_{\mathbb{Z}_p} \mathcal{O}_E \cong \langle (1,0), (0, 1) \rangle_{\mathcal{O}_E},
\]
which implies that $k=0$ in Theorem \ref{main thm of classification lambda=2}.
By Lemma \ref{Fujii}, we have $L_{K} \cap \widetilde{K}=K_{1}^{{\rm an}}$.
Indeed, we have
$(I(3)/S(3^4)) \otimes \mathbb{Z}_3 \cong \mathbb{Z}/3\mathbb{Z}
\oplus \mathbb{Z}/3^{3}\mathbb{Z} \oplus \mathbb{Z}/3^{4} \mathbb{Z}.$
Hence  ${\rm Gal}(L_K \cap \widetilde{K}/K)$ is a direct summand of ${\rm Gal}(L_K/K)$.
Using \cite[Theorem $2$]{Br}, we obtain
\begin{eqnarray*}
x^6 - 20 x^4 + 100 x^2 + 38992
\end{eqnarray*}
as a defining polynomial of $K_{1}^{{\rm an}}$ over $\mathbb{Q}$.
We can check that both $53$ and $251$ are primes 
which split completely in $K_{1}^{{\rm c}}/\mathbb{Q}$.
We put
\begin{eqnarray*}
\frak{q} &=& (251,\; -18+\sqrt{-2437}),\\
\frak{l} &=& (53,\; -1+ \sqrt{-2437}),
\end{eqnarray*}
which are prime ideals in $K$ lying above $251, 53$, respectively.
Using PARI/GP, we compute prime ideals 
$\frak{Q}_{i}, \overline{\frak{Q}_i}, \frak{L}_i$, $\overline{\frak{L}_{i}}$
in $\mathcal{O}_{K_1^{{\rm c}}}$
($i=1,2,3$)
which satisfy
\begin{eqnarray*}
251 \mathcal{O}_{K_1^{{\rm c}}} &=& \frak{Q}_1 \overline{\frak{Q}_1} \cdots \frak{Q}_3 \overline{\frak{Q}_3},\\
53  \mathcal{O}_{K_1^{{\rm c}}} &=& \frak{L}_1 \overline{\frak{L}_1} \cdots \frak{L}_3 \overline{\frak{L}_3}
\end{eqnarray*}
and 
$\frak{Q}_{i} ~|~ \frak{q}$, 
$\overline{\frak{Q}_i} ~|~ \overline{\frak{q}}$, 
$\frak{L}_{i} ~| ~\frak{l}$,
$\overline{\frak{L}_i} ~|~ \overline{\frak{l}}$
for $i=1,2,3.$
We also compute 
$$
\mathrm{Cl}_{K_1^{{\rm c}}} = \mathbb{Z}/(434 \cdot 9)\mathbb{Z}~ [\frak{c}_1] \oplus \mathbb{Z}/9\mathbb{Z}~ [\frak{c}_2]
$$
for some ideals $\frak{c}_1$ and $\frak{c}_2$ in $\mathcal{O}_{K_1^{{\rm c}}}$,
which was computed by PARI/GP.
Pick one of $\frak{Q}_{i} ~|~ \frak{q}$ (resp. $\frak{L}_{i} ~| ~\frak{l}$) and we may assume that it is $\frak{Q}_{1}$ (resp. $\frak{L}_{1}$).
As in \S \ref{A method of computing mu_{21} and mu_{22}}, we take a basis $\{[434\frak{Q}_1], [434\frak{L}_1] \}$ of $A_{K_1^{{\rm c}}}$, in other words,
\[
A_{K_1^{{\rm c}}} = \mathbb{Z}/9 \mathbb{Z} ~[434\frak{Q}_1] \oplus \mathbb{Z}/9 \mathbb{Z} ~[434\frak{L}_1].
\]
This implies that both $s$ and $t$ in \S \ref{A method of computing mu_{21} and mu_{22}} are $434$.

Now, to obtain a representation as (\ref{overline{S} ([ufrak{Q}_1+vfrak{L}_1])}),
we consider the Galois action of $\mathrm{Gal}(K_1^{{\rm c}}/K)$ to $[\frak{Q}_1]$ and $[\frak{L}_1]$.
Write $[\frak{Q}_1]$ and $[\frak{L}_1]$ as linear forms of $[\frak{c}_1]$ and $[\frak{c}_2]$:
\[
[\frak{Q}_1] = 3229 [\frak{c}_1] + 6[\frak{c}_2],\quad
[\frak{L}_1] = 2580 [\frak{c}_1] + 7[\frak{c}_2].
\]
On the other hand, we can compute 
$$
\mathrm{Gal}(L_{K}/K_1^{{\rm an}})
=
\left\langle \left( \frac{L_K/K}{\frak{q}} \right) \cdot \left( \frac{L_K/K}{\frak{l}} \right) \right\rangle
=
\left\langle \left( \frac{L_K/K}{\frak{q}} \right)^{434} \cdot \left( \frac{L_K/K}{\frak{l}} \right)^{434} \right\rangle,
$$
in other words, both $u$ and $v$ in \S \ref{A method of computing mu_{21} and mu_{22}} are $434$.
Let
$\overline{\sigma}$ be a generator of $\mathrm{Gal}(K_1^{{\rm c}}/K)$,
which was computed by PARI/GP.
We do not write down $\overline{\sigma}$ because it is complicated.
Then, by computation of $\overline{\sigma}[\frak{c}_1]$ and $\overline{\sigma}[\frak{c}_2]$, we can write $\overline{\sigma} [\frak{Q}_1]$ and $\overline{\sigma} [\frak{L}_1]$ as linear forms of $[\frak{c}_1]$ and $[\frak{c}_2]$:
\begin{eqnarray*}
\overline{\sigma} [\mathfrak{Q}_1]&=& 1327 [\mathfrak{c}_1] + 3 [\mathfrak{c}_2],  \\
\overline{\sigma} [\mathfrak{L}_1]&=& 624 [\mathfrak{c}_1] +  [\mathfrak{c}_2]. 
\end{eqnarray*}
Take a topological generator ${\sigma} \in \mathrm{Gal}(K_\infty^{{\rm c}}/K)$ such that ${\sigma}$ is an extension of $\overline{\sigma}$.
Then in the same way as Example \ref{example of 1.2(iii)(iv)}, we regard $X_{K_\infty^{{\rm c}}}$ as a $\Lambda$-module 
by the isomorphism (\ref{isom ring}).
We note that
$E$, $\mathrm{ord}_E(\alpha)$, $\mathrm{ord}_E(\beta)$, and $\mathcal{M}_{f(S)}^E$ do not depend on the choice
of ${\sigma}$.
Since $\mathbb{Z}_p[\mathrm{Gal}(K_1^{{\rm c}}/K)] \cong \Lambda/\omega_1(S) \Lambda$, we get
\begin{eqnarray*}
\overline{S}[\mathfrak{Q}_1]&=&
- \frac{5574}{7123} [\mathfrak{Q}_1]
+ \frac{1725}{7123}[\mathfrak{L}_1],  \\
\overline{S}[\mathfrak{L}_1]&=& 
\frac{1788}{7123} [\mathfrak{Q}_1] 
- \frac{7638}{7123}  [\mathfrak{L}_1], 
\end{eqnarray*}
where $\overline{S}=S$ mod $\omega_1(S)$.
Using the commutative diagram before Theorem \ref{main thm}, 
we can take $x_1, x_2 \in X_{K_\infty^{{\rm c}}}$
such that
\[
\psi_{1}(x_1 \mathrm{~mod~} \omega_{1}(S)) = [434\frak{Q}_{1}],\quad
\psi_{1}(x_2 \mathrm{~mod~} \omega_{1}(S)) = [434\frak{Q}_1+434\frak{L}_1].
\]
These implies that (\ref{overline{S} ([ufrak{Q}_1+vfrak{L}_1])}) becomes
\[
Sx_2 \mathrm{~mod~} \omega_1(S) = -\frac{923118}{7123} x_1 - \frac{1643124}{7123} x_2 \mathrm{~mod~} \omega_1(S).
\]
We note that $\displaystyle{\mathrm{ord}_E\left(\frac{923118}{7123} \right)=1, \mathrm{ord}_E\left(\frac{1643124}{7123}\right)=1}$.
By Theorem \ref{main thm}, we obtain $\mathrm{ord}_E(\mu_{21})$ $=\mathrm{ord}_E(\mu_{22})$ $=0$.
Therefore $X_{\widetilde{K}}$ is a cyclic $\mathbb{Z}_p[[\mathrm{Gal}(\widetilde{K}/K)]]$-module
by Theorem \ref{main thm of classification lambda=2} (iv).
\end{rm}
\end{Example}
By the same method as in Example \ref{example of 1.2(iv)}, we obtain Tables 5 and 6.
\begin{table}[htb]
\caption{}
\begin{tabular}{|c||c|c|c|c|c|c|c|c|} \hline
$d$&
\scriptsize{$\mathrm{ord}_E(\alpha-\beta)$}&
$k$&
$m$&
\scriptsize{$L_{K} \ \cap \widetilde{K}$}&
$E/\mathbb{Q}_3$&
$A_{0}$&
$X_{\widetilde{K}}$\\
\hline \hline
2437&
1&
0&
1&
$K_{1}^{{\rm an}}$&
\footnotesize{unramified}&
$(3,3)$&
cyclic\\
3886&
1&
0&
1&
$K_{1}^{{\rm an}}$&
\footnotesize{$E=\mathbb{Q}_p$}&
$(3,3)$&
cyclic\\
4027&
1&
0&
1&
$K_{1}^{{\rm an}}$&
\footnotesize{$E=\mathbb{Q}_p$}&
$(3,3)$&
cyclic\\
7977&
1&
0&
1&
$K_{1}^{{\rm an}}$&
\footnotesize{unramified}&
$(3,3)$&
cyclic\\
\hline
\end{tabular}
\\[2pt]
(The integer $k$ is defined by (\ref{definition of k}) and $m={\rm min}\{{\rm ord}_E(\alpha), {\rm ord}_E(\beta)\}$)
\end{table}
\begin{table}[htb]
\caption{}
\begin{tabular}{|c||c|} \hline
$d$&
Defining polynomial of $K_{1}^{{\rm an}}$\\
\hline \hline
2437&
$x^6 - 20 x^4 + 100 x^2 + 38992$\\
3886&
$x^6 - 66x^4 + 1089x^2 + 62176$\\
4027&
$x^6 - 44x^4 + 484x^2 + 4027$\\
7977&
$x^6 - 2x^5 - 53x^4 + 126x^3 + 8634x^2 - 1944x + 1296$\\
\hline
\end{tabular}
\end{table}

\vspace{10pt}

{\bf  Acknowledgement.}
The authors would like to express their sincere gratitude to Professor Masato Kurihara.
They started a series of studies progressed in this paper and the previous one \cite{MOMO}, 
drawing their inspiration from his brilliant idea in \cite{Kurihara11} 
appearing in the argument of reducing the refined class number formula to the Gross' conjecture.
He always encouraged the authors, gave helpful suggestions, and kindly answered many questions.
%
The authors also would like to express their thanks to Professor Satoshi Fujii for his useful comments, 
to Professor Takashi Fukuda for introducing to the authors useful functions in PARI/GP \cite{PARI/GP}, one of which computes the characteristic polynomials of Iwasawa modules.


\vspace*{10pt}
\noindent
Takashi MIURA,
\\
Department of Creative Engineering,
National Institute of Technology, Tsuruoka College,
\ 
104 Sawada, Inooka, Tsuruoka, Yamagata 997-8511, Japan.
\\
{\tt 
t-miura@tsuruoka-nct.ac.jp
}
\\[10pt]
Kazuaki MURAKAMI,
\\
Department of Mathematical Sciences, Graduate School of Science and Engineering, Keio University,
\ 
Hiyoshi, Kohoku-ku, Yokohama, Kanagawa 223-8522, Japan.
\\
{\tt 
murakami\_0410@z5.keio.jp
}
\\[10pt]
Keiji OKANO,
\\
Department of Teacher Education,
\ 
3-8-1 Tahara, Tsuru-shi, Yamanashi 402-0054, Japan.
\\
{\tt 
okano@tsuru.ac.jp
}
\\[10pt]
Rei OTSUKI,
\\
Department of Mathematics,
Keio University,
\ 
3-14-1 Hiyoshi, Kouhoku-ku, Yokohama 223-8522, Japan.
\\
{\tt 
ray\_otsuki@math.keio.ac.jp
}

\end{document}